\documentclass[final,1p,times]{elsarticle}
\usepackage{amsmath}
\allowdisplaybreaks[4]
\usepackage{amsthm}
\usepackage{amscd}
\usepackage{amsmath}
\usepackage{amsfonts}
\usepackage{amssymb}
\usepackage{graphicx}
\newtheorem{theorem}{Theorem}[section]

\newtheorem{lemma}[theorem]{Lemma}
\newtheorem{remark}{Remark}

\usepackage{mathrsfs}
\usepackage{titletoc}
\numberwithin{equation}{section}

\renewcommand{\d}{\delta}

\newcommand{\D}{\Delta}
\newcommand{\ra}{\rightarrow}

\newcommand{\f}{\frac}
\newcommand{\al}{\alpha}

\renewcommand{\l}{\lambda}
\renewcommand{\L}{\Lambda}
\newcommand{\be}{\begin{equation}}
\renewcommand{\ra}{\rightarrow}
\newcommand{\ee}{\end{equation}}
\newcommand{\bea}{\begin{eqnarray}}
\newcommand{\eea}{\end{eqnarray}}
\newcommand{\bna}{\begin{eqnarray*}}
\newcommand{\ena}{\end{eqnarray*}}
\renewcommand{\o}{\omega}

\newcommand{\iv}{\int_{V}}

\newcommand{\me}{\mathrm{e}}

\renewcommand{\le}{\left}
\newcommand{\ri}{\right}

\newcommand{\ve}{\vert}
\newcommand{\V}{\Vert}

\newcommand{\na}{\nabla}
\newcommand{\upl}{\upsilon'_{\lambda}}
\newcommand{\up}{\upsilon}

\newcommand{\va}{\varphi}
\newcommand{\wl}{w_{\lambda}}
\journal{***}
\begin{document}

\begin{frontmatter}

\title{ Existence theorems for a generalized  Chern-Simons equation on finite graphs}
\author{Jia Gao}
\ead{s20213101929@cau.edu.cn}
\address{Department of Applied Mathematics, College of Science, China Agricultural University,  Beijing, 100083, P.R. China}
\author{Songbo Hou \corref{cor1}}
\ead{housb@cau.edu.cn}
\address{Department of Applied Mathematics, College of Science, China Agricultural University,  Beijing, 100083, P.R. China}

\cortext[cor1]{Corresponding author: Songbo Hou}

\begin{abstract}
Denote by  $G=(V,E)$ a finite graph.   We study  a generalized Chern-Simons equation
\begin{eqnarray*}
\Delta u=\lambda \mathrm{e}^u(\mathrm{e}^{bu}-1)+4\pi\sum\limits_{j=1}^{N}\delta_{p_j}
\end{eqnarray*}
on $G$, where $\lambda$  and $b$ are  positive constants; $N$ is a positive integer;  $p_1, p_2, \cdot\cdot\cdot, p_N$ are distinct vertices of $V$
and $\delta_{p_j}$ is  the Dirac delta mass at $p_j$. We prove that there exists a critical value $\lambda_c$ such that the equation has a solution if $\lambda\geq \lambda_c$ and the equation has no solution if $\lambda<\lambda_c$. We also prove that if  $\lambda>\lambda_c$ the equation has at least two solutions which include  a local minimizer for the corresponding functional and a mountain-pass type
solution. Our results extend and complete those of Huang et al. (Comm. Math. Phys. 377(1) (2020)  613-621), Hou et al. (Calc. Var. Partial Dif. 61(4) (2022), Art.139).

\end{abstract}


 \end{frontmatter}

\section{Introduction}

The study on the  Abelian Chern-Simons model  is an interesting topic. In this model, the Yang-Mills or
the Maxwell field term are replaced by the  specific Higgs potential.  The  Chern-Simons gauge field yields that vortices  are magnetically and elecrically charged. A natural question is how to get the existence of condensates or periodic multivortices. Under the assumption that the Higgs potential takes a sextic form, Caffarelli and Yang \cite{CY95} reduced the self-dual Chern-Simons equations to the following form

\be\label{cs1}
\D u=\l \me^u(\me^{u}-1)+4\pi\sum\limits_{j=1}^{N}\delta_{p_j}
\ee
on $\Omega$,
where $\l>0$, $\Omega\subset \mathbb{R}^2$ is a doubly periodic domain, $\delta_p$ is the Dirac distribution concentrated at $p\in \Omega$.
Then they got the existences of solutions by the rigorous
mathematical analysis.
The solutions of Eq.(\ref{cs1}) are of  great significance in many fields of physics. The topological solutions, non-topological
solutions, and solutions over a doubly periodic domain  have been extensively studied \cite{W91, SY92,SY95, Ch07, TG07,CKL}.

Chen and Han \cite{SH13} studied a generalized Chern-Simoms equation

\be\label{csg}
\D u=\l \me^u(\me^{bu}-1)+4\pi\sum\limits_{j=1}^{N}\delta_{p_j},\,\,\,\,\,\,x\in \Omega,
\ee
where $\l>0$, $b>0$, $\Omega\subset \mathbb{R}^2$ is a doubly periodic domain, $\delta_p$ is the Dirac distribution concentrated at $p\in \Omega$.
They extended the results in \cite{CY95}. As an application, they solved an open problem about the existence of solutions of the equation arising in the self-dual Maxwell-Chern-Simons model \cite{CC99}.

In this paper, we consider the discrete form of Eq.(\ref{csg}) on a finite graph.  A  finite graph $G=(V,E)$ is composed  of vertices $V $ and edges $E$.  We assume that $G$ is connected throughout the paper.  We call $G$ is weighted if each edge $xy\in E$ is assigned a weight $\o_{xy}$ which is positive and symmetric. Let $\mu :V\ra \mathbb{R}^{+}$ be a finite measure.  If vertex $y$ is adjacent
to vertex $x$,  we write  $y\sim x$. For any function $u:V\ra \mathbb{R}$, we define the
$\mu$-Laplace operator   by
$$\D u(x)=\f{1}{\mu(x)}\sum\limits_{y\sim x}\omega_{xy}(u(y)-u(x)).$$
The associated gradient form reads
\be\label{lpd}\Gamma(u,\up)=\f{1}{2\mu(x)}\sum\limits_{y\sim x}\omega_{xy}(u(y)-u(x))(\up(y)-\up(x)).\ee
Write $\Gamma(u)=\Gamma(u,u)$. The length of its gradient  is denoted by
$$\ve \na u\ve (x) =\sqrt {\Gamma (u)(x)}=\le( \f{1}{2\mu(x)}\sum\limits_{y\sim x}\omega_{xy}\le(u(y)-u(x)\ri)^2\ri)^{\f{1}{2}}.$$
For any  function $f:V\ra \mathbb{R}$,  the integral of $f$  over $V$ is defined by
$$\int_V fd\mu=\sum\limits_{x\in V}\mu(x)f(x).$$
Define the Sobolev space $$W^{1,2}(V)=\le\{u\,\Big|\, u:V\rightarrow \mathbb{R},\,\iv\le(\ve \na u\ve^2 +u^2\ri)d\mu<+\infty\ri\}.$$  The  norm of $u\in W^{1,2}(V)$ is

$$\V u\V_{W^{1,2}(V)}=\le(\iv\le(\ve \na u\ve^2 +u^2\ri)d\mu\ri)^{1/2}.$$

We consider the  generalized Chern-Simons equation derived in  \cite{SH13}, i.e.,
\be\label{cs}
\D u=\l \me^u(\me^{bu}-1)+4\pi\sum\limits_{j=1}^{N}\delta_{p_j}
\ee
on a graph $G=(V,E)$,  where $\l$ and $b$ are   positive constants; $N$ is a positive integer;  $p_1, p_2, \cdot\cdot\cdot, p_N$ are distinct vertices of $V$
and $\d_{p_j}$ is  the Dirac delta mass at $p_j$.

We prove the following existence results for Eq.(\ref{cs} ) on $G$.
\begin{theorem}
There exists a critical value $\l_c$ depending on the graph satisfying
$$\l_c\geq  \f{(b+1)^{b+1}}{b}\f{4\pi N}{\ve V\ve},$$
such that
\begin{enumerate}[(i)]
\item   Eq.(\ref{cs}) has a solution if $\l\geq \l_c$, however,   Eq.(\ref{cs}) has no solution if $\l<\l_c$.

\item   Eq.(\ref{cs}) has at least two solutions if $\l>\l_c$.
\end{enumerate}

\end{theorem}
\begin{remark}
If $b=1$, then Eq.(\ref{cs}) becomes Eq.(\ref{cs1}). There also exist  multiple solutions for Eq.(\ref{cs1}) on graphs. This completes the results of Huang et al.\cite{HLY}, Hou et al. \cite{HS}.
\end{remark}
There are a lot of research on theories of partial differential equations on graphs. For the discrete Chern-Simons equations,
Huang, Lin and Yau \cite{HLY} studied Eq.(\ref{cs1}) on graphs and got the existence results except the critical case. We solved  the critical case \cite{HS}, and   got  the existence results for a generalized Chern-Simons-Higgs equation, which was also studied in \cite{LZ}. Huang, Wang and Yang \cite {HWY} considered the
relativistic Abelian Chern-Simons equations and established the
existence of multiple solutions. Related results also include \cite{CH, CHS}.  The results for  Yamabe type equations include \cite{GLY16, GJJ, ZL18, GJ19, LZ19, ZL19}. The results for Kazdan-Warner equations include \cite{GLY, GJ18, ZC18, Ge20, LY20}. One may refer to \cite{HSZ, Hou21} for the biharmonic equations.

  The rest of the paper is arranged as follows. In Section Two, we first use the upper and lower solutions method to get the existence of the single solution for the non-critical case. Then we use the  prior estimates to prove the existence for the critical case.  In Section Three, we use the mountain pass theorem to get the existence of  multiple solutions.

\section{Single Solution}

Denote $u=u_0+\up$, where $u_0$ is the solution of the equation
\be \label{up1}\D u= -\f{4\pi N}{\ve V\ve}+4\pi \sum\limits_{j=1}^{N}\delta_{p_j}.\ee
Then we can reduce Eq.(\ref{cs}) into
\be\label{cs2}
\D \up =\l\me^{u_0+\up}\le(\me^{b(u_0+\up)}-1\ri)+\f{4\pi N}{\ve V\ve}.
\ee
We introduce the definition of upper and lower solutions.
If a function $\up_+$ satisfies $$\D \up_+ \leq \l\me^{u_0+\up_+}\le(\me^{b(u_0+\up_+)}-1\ri)+\f{4\pi N}{\ve V\ve}$$ on $V$, we call it an upper solution of (\ref{cs2}). Analogously, a function $\up_-$ is said to be  a lower solution of (\ref{cs2}) if it satisfies
\be\label{sub}
\D \up_- \geq \l\me^{u_0+\up_-}\le(\me^{b(u_0+\up_-)}-1\ri)+\f{4\pi N}{\ve V\ve}
\ee
on $V$.
We introduce Lemma 4.1 in \cite{HLY}, which is  referred to as the maximum principle.
\begin{lemma}

Let $G=(V,E)$ be a finite graph.
If for some constant $K>0$, $\D  u(x) -Ku(x) \geq 0$ for all $x\in V$, then $u\leq 0$ on $V$ .
\end{lemma}

Set $\up_0=-u_0$.  Let $\{\up_n\}$ be the iterative sequence determined by the scheme
\be \label{cs3}
\left\{
\begin{aligned}&\le(\D-K\ri)\up_n=\l\me^{u_0+\up_{n-1}}\le(\me^{b(u_0+\up_{n-1})}-1\ri)-K\up_{n-1}+\f{4\pi N}{\ve V\ve},\\
&n=1,2,...,
\end{aligned}
\right.
\ee
where $K$ denotes a positive constant.

\begin{lemma}
 If $K>b\l$, then the sequence $\{\up_n\}$  satisfies
$$\up_0> \up_1> \up_2>\cdot\cdot\cdot> \up_n>\cdot\cdot\cdot
>\up_{-},$$  where $\up_-$ is any lower solution of (\ref{cs2}). Hence,  if (\ref{cs2}) admits a lower solution,  $\{\up_n\}$ converges to a solution of (\ref{cs2}), which is the maximal solution.
 \end{lemma}
 \begin{proof}

 We use the inductive method.
  By (\ref{cs3}),  we obtain
 $$\le(\D-K\ri)\up_1=-K\up_0+\f{4\pi N}{\ve V\ve}, $$  which together with (\ref{up1}) yields
 \be\label{ps1}\le(\D-K\ri)(\up_1-\up_0)=4\pi \sum\limits_{j=1}^{N}\delta_{p_j}\geq 0.\ee By Lemma 2.1, we get
 $\up_1(x)\leq \up_0(x)$ for all  $x\in V$.
 Furthermore, we can prove $\up_1<\up_0$ on $V$ by contradiction.
 Suppose $(\up_1-\up_0)(x_0)=0$ for some $x_0\in V$. Obviously,  $x_0$ is the maximum point of $\up_1-\up_0$ on $V$. In view of (\ref{ps1}),  we have
 $$\D(\up_1-\up_0)(x_0)\geq K(\up_1-\up_0)(x_0)\geq 0. $$
 Noting the definition of the $\mu$-Laplace, we conclude that if $x\sim x_0$, $(\up_1-\up_0)(x)=(\up_1-\up_0)(x_0)$.  Since $G$ is connected, we infer that $(\up_1-\up_0)(x)=(\up_1-\up_0)(x_0)$ for any $x\in V$, which contradicts (\ref{ps1}).  Hence we have  $\up_1<\up_0$ on $V$.

 Suppose that $$\up_0>\up_1>\cdot\cdot\cdot> \up_k.$$
By  (\ref{cs3}) and noting $K>b\l$, we calculate
\bna\begin{aligned}
(\D-K)(\up_{k+1}-\up_k)&=\l\le[\me^{u_0+\up_k}\le(\me^{b(u_0+\up_k)}-1\ri)-\me^{u_0+\up_{k-1}}\le(\me^{b(u_0+\up_{k-1})}-1\ri)\ri]-K(\up_k-\up_{k-1})\\
&=\l\me^{\xi}\le[(b+1)\me^{b\xi}-1\ri](\up_k-\up_{k-1})-K(\up_k-\up_{k-1})\\
&\geq (b\l-K)(\up_k-\up_{k-1})\\
&>0,
\end{aligned}
\ena
 where the second equality follows from  the mean value theorem, $u_0+\up_k\leq \xi\leq u_0+\up_{k-1}$, and
 we have used $\me^{\xi}\leq \me^{\up_{k-1}+u_0}\leq 1$, $\me^{b\xi}\leq 1$.
It follows from Lemma 2.1 that  $\up_{k+1}\leq\up_k$. Applying the same method  as in proving $\up_1<\up_0$, we obtain  $\up_{k+1}<\up_k$.
Hence, we establish $$\up_0> \up_1> \up_2>\cdot\cdot\cdot> \up_n>\cdot\cdot\cdot.$$

 We will also  prove that $\up_k>\up_-$ for all $k\geq 0$ by induction.  Noting (\ref{cs3}) and the definition of the lower solution,   we have
 \be\label{ps2}\begin{aligned}
\D(\up_--\up_0)&\geq \l\me^{\up_--\up_0}(\me^{b(\up_--\up_0)}-1)+4\pi \sum\limits_{j=1}^{N}\delta_{p_j}\\
&=\l b\me^{\up_--\up_0}\me^{b\xi}(\up_--\up_0)+4\pi \sum\limits_{j=1}^{N}\delta_{p_j},
\end{aligned}
\ee
where $\xi$ lies between $\up_--\up_0$ and $0$.   Denote $\up_-(x_0)-\up_0(x_0)=\max_{x\in V}\{\up_-(x)-\up_0(x)\}$ for some $x_0\in V$. Supposing $(\up_-(x_0)-\up_0(x_0))\geq 0$, then by (\ref{ps2}), we have
$$\D(\up_--\up_0)(x_0)\geq 0.$$
It yields that if $x\sim x_0$, then $(\up_--\up_0)(x)=(\up_--\up_0)(x_0)$.  Furthermore,  $(\up_--\up_0)(x)\equiv (\up_--\up_0)(x_0)$  for all $x\in V$ since $G$  is  connected.  Hence there is a contradiction with (\ref{ps2}) at $p_j$. It holds that  $\up_0>\up_-$ on $V$.

 Suppose that for some $k\geq 0$, $\up_k> \up_-$. In view of (\ref{sub}), (\ref{cs3}) and the fact $K>b\l$, we obtain
\bna\begin{aligned}
(\D-K)(\up_--\up_{k+1})&\geq \l\le[ \me^{\up_--\up_0}\le(\me^{b(\up_--\up_0)}-1\ri)-\me^{\up_k-\up_0}\le(\me^{b(\up_k-\up_0)}-1\ri)\ri]-K(\up_--\up_{k})\\
&=\l\me^{\xi}\le[(b+1)\me^{b\xi}-1\ri](\up_--\up_{k})-K(\up_--\up_{k})\\
&\geq (b\l-K)(\up_--\up_{k})\\
&>0,
\end{aligned}
\ena
where we have used the mean value theorem, $\up_--\up_0\leq \xi\leq \up_k-\up_0$. We have $\up_{k+1}>\up_-$  by taking the same procedure  as before.   Combining (\ref{cs3}) and the monotonicity of $\{\up_n\}$, we conclude that $\{\up_n\}$  converges to a solution of  (\ref{cs2}) pointwisely.  Such a solution is bigger than any other solution. This finishes the proof of Lemma 2.2.
\end{proof}
In order to get the existence of the solution of Eq.(\ref{cs2}), we only need to find a lower solution. We have the following lemma.
\begin{lemma}
If  $\l$ is sufficiently large, there is a solution of Eq.(\ref{cs2}).
\end{lemma}

\begin{proof}
Since $G$ is finite, $u_0$ is a bounded function.  There exists a positive constant $c$ such that $u_0-c<0$. Set $\up_-=-c$. We see that if $\l$ is sufficiently large, then
$$\D \up_- =0> \l\me^{u_0-c}\le(\me^{b(u_0-c)}-1\ri)+\f{4\pi N}{\ve V\ve}.$$
This implies that $\up_-$ is a lower solution of (\ref{cs2}).
 Consequently,  Eq.(\ref{cs2}) admits a solution. We complete the proof.
\end{proof}
 Define $f(t)=\me^t(\me^{bt}-1)$, $t\in (-\infty, 0]$. A simple computation implies that $f$ has a unique minimum $-\f{b}{(b+1)^{b+1}}$.
 If $\up$ is a solution of Eq.(\ref{cs2}), it is also a lower solution of Eq.(\ref{cs2}). Then Lemma 2.2 implies that $u_0+\up<0$.
 Thus,  we obtain
 $$\D \up \geq -\f{b\l}{(b+1)^{b+1}}+\f{4\pi N}{\ve V\ve}.$$
 Integrating both sides of the above inequality over $V$, we get
\be\label{ncc}
\l\geq \f{(b+1)^{b+1}}{b}\f{4\pi N}{\ve V\ve}.
\ee
 If Eq.(\ref{cs2}) has a solution, then $\l$ must satisfy (\ref{ncc}).
 \begin{lemma}\label{set}
 There is a critical value $\l_c$ satisfying
 $$\l_c\geq \f{(b+1)^{b+1}}{b}\f{4\pi N}{\ve V\ve},$$
  such that Eq.(\ref{cs2}) admits a solution for $\l>\l_c$,  while Eq.(\ref{cs2})  admits no solution for  $\l<\l_c$.
 \end{lemma}
 \begin{proof}
 Denote
 $$\Lambda=\Big\{\l>0\,\big|  \text{ $\l $ is  such  that  (\ref{cs2})  has a solution}\Big\}.$$
 We claim that $\L$ is an interval. In fact, if $\l'\in \Lambda$, we can show that $[\l',+\infty)\subset \L.$  Let $\up'$ be the solution of
 Eq.(\ref{cs2}) with $\l=\l'$.  If $\l>\l'$, in view of $u_0+\up'<0$, we have
 $$\D \up'=\l'\me^{u_0+\up'}\le(\me^{b(u_0+\up')}-1\ri)+\f{4\pi N}{\ve V\ve} > \l\me^{u_0+\up'}\le(\me^{b(u_0+\up')}-1\ri)+\f{4\pi N}{\ve V\ve},$$
 which implies that $\up'$ is a lower solution of Eq.(\ref{cs2}) for $\l>\l'$. Therefore, Eq.(\ref{cs2}) has a solution if $\l>\l'$ and $ [\l',+\infty)\subset \Lambda$.

 Define
 $\l_c=\inf\{\l\,|\,\l\in \L\}$.  Noting (\ref{ncc}),  we obtain
 $$\l_c\geq \f{(b+1)^{b+1}}{b}\f{4\pi N}{\ve V\ve}$$ by taking the limit $\l\ra \l_c$.
 \end{proof}

Now,  we deal with the critical case $\l=\l_c$. In the following lemma, we  get the monotonicity of the solution of (\ref{cs2}) with respect to $\l$, which will be used later.
 \begin{lemma}
 Let $\{\up_{\l}\,|\,\l>\l_c\}$ be the family of maximal solutions of (\ref{cs2}). Then there holds $\up_{\l_1}>\up_{\l_2}$ if $\l_1>\l_2>\l_c$.
\end{lemma}
\begin{proof}
Assume that $\l_1>\l_2$ and the associated solutions are $\up_{\l_1}$, $\up_{\l_2}$ respectively.   Then by (\ref{cs2}), we have

\bna\begin{aligned}
\D \up_{\l_2}&=\l_2\me^{u_0+\up_{\l_2}}\le(\me^{b(u_0+\up_{\l_2})}-1\ri)+\f{4\pi N}{\ve V\ve}\\
&= \l_1\me^{u_0+\up_{\l_2}}\le(\me^{b(u_0+\up_{\l_2})}-1\ri)+\f{4\pi N}{\ve V\ve}\\
&\,\,\,\,\,+(\l_2-\l_1)\me^{u_0+\up_{\l_2}}\le(\me^{b(u_0+\up_{\l_2})}-1\ri)\\
&\geq \l_1\me^{u_0+\up_{\l_2}}\le(\me^{b(u_0+\up_{\l_2})}-1\ri)+\f{4\pi N}{\ve V\ve}.
\end{aligned}\ena
Thus,  $\up_{\l_2}$ is a lower solution of (\ref{cs2}) with $\l=\l_1$.  Then  by Lemma 2.2, we obtain  $\up_{\l_1}\geq \up_{\l_2}$.
Furthermore,
\be\label{mo1}\begin{split}
\D (\up_{\l_2}-\up_{\l_1})&=\l_2\me^{u_0+\up_{\l_2}}\le(\me^{b(u_0+\up_{\l_2})}-1\ri)-\l_1\me^{u_0+\up_{\l_1}}\le(\me^{b(u_0+\up_{\l_1})}-1\ri)\\
&=\l_1\me^{u_0+\up_{\l_2}}\le(\me^{b(u_0+\up_{\l_2})}-1\ri)-\l_1\me^{u_0+\up_{\l_1}}\le(\me^{b(u_0+\up_{\l_1})}-1\ri)\\
&\,\,\,\,\,+(\l_2-\l_1)\me^{u_0+\up_{\l_2}}\le(\me^{b(u_0+\up_{\l_2})}-1\ri)\\
&>\l_1\le[\me^{u_0+\up_{\l_2}}\le(\me^{b(u_0+\up_{\l_2})}-1\ri)-\me^{u_0+\up_{\l_1}}\le(\me^{b(u_0+\up_{\l_1})}-1\ri)\ri] \\
&=\l_1\me^{\xi}\le[(b+1)\me^{b\xi}-1\ri]\le(\up_{\l_2}-\up_{\l_1}\ri)\\
&\geq b\l_1(\up_{\l_2}-\up_{\l_1}),
\end{split}
\ee
where $u_0+\up_{\l_2}\leq \xi\leq u_0+\up_{\l_1}$.
Suppose that $\up_{\l_2}-\up_{\l_1}$ achieves  0 at some point $x_0\in V$.  Then by (\ref{mo1}),  we get
 $$\D(\up_{\l_2}-\up_{\l_1})(x_0)>0.$$ However, $\D(\up_{\l_2}-\up_{\l_1})(x_0)\leq 0$ according to the definition of the $\mu$-Laplace.
  There is a contradiction. Hence,  $\up_{\l_{2}}-\up_{\l_1}<0$ and we prove that $\up_{\l_1}>\up_{\l_2}$ if $\l_1>\l_2>\l_c$.
\end{proof}

Now, we estimate the bound of solutions of (\ref{cs2}) in $W^{1,2}(V)$.
We first introduce the  Poincar\'{e} inequality and the  Trudinger-Moser inequality which were proved in \cite{GLY}.
\begin{lemma}(the Poincar\'{e} inequality)

Let $G=(V,E)$ be a finite graph. Then there exists some constant $C$ depending only on $G$ such that
$$\iv u^2d\mu\leq C\iv \ve \na u\ve^2d\mu,$$ for all $u\in V^{\mathbb{R}}$ with $\iv ud\mu=0$, where $V^{\mathbb{R}}=\{u\,|\,u \,\text{is a real function}: V\ra\mathbb{R}\}$.
 \end{lemma}
\begin{lemma}\label{lm3}  (the Trudinger-Moser inequality)

Let $G=(V,E)$ be a finite graph. For any $\al>0$ and all functions $u\in V^{\mathbb{R}}$ with $\iv \ve \na u\ve^2 d\mu\leq 1$ and $\iv ud\mu=0$, there exists some constant $C>0$ depending only on $\al$
and $G$ such that
$$\iv \me^{\al u^2}d\mu\leq C.$$
\end{lemma}
Applying the above lemmas and the integration  method, we get the following lemma.
\begin{lemma}
For any maximal solution $\up_{\l}$ of (\ref{cs2}),   we split it into two parts, namely, $\up_{\l}=\bar{\up}_{\l}+\up'_{\l}$, where $\bar{\up}_{\l}=\f{1}{\ve V\ve}\int_{V}\up_{\l} d\mu$ and $\up'_{\l}=\up_{\l}-\bar{\up}_{\l}$. Then $\up'_{\l}$ satisfies
$$\V \na\up'_{\l}\V_2\leq C\l, $$
where $C$ is a positive constant depending only on $\ve V\ve$.  Furthermore, we  obtain
$$\ve \bar{\up}_{\l}\ve\leq C(1+\l+\l^2)$$
and
$$\V \up_{\l}\V_{W^{1,2}(V)}\leq C(1+\l+\l^2).$$
\end{lemma}
\begin{proof}
In the following, we use $C$ to denote a positive constant which may vary from line to line.  Noting $\up_{\l}=\bar{\up}_{\l}+\up'_{\l}$, we multiply both sides of Eq.(\ref{cs2}) by $\upl$. Then integrating over $V$ and using the Poincar\'{e} inequality, we have
\bna\begin{aligned}
\V\na\up'_{\l}\V_2^2&=\l\iv \me^{u_0+\up_{\l}}\le(1-\me^{b(u_0+\up_{\l})}\ri)\upl d\mu\\
&\leq \l\iv \ve\up'_{\l}\ve d\mu\leq C\l\ve V\ve^{1/2}\V\na\upl\V_2,
\end{aligned}\ena
which yields
\be\label{na}
\V\na \up'_{\l}\V_2\leq C\l.
\ee
In view of $u_0+\up_{\l}=u_0+\bar{\up}_{\l}+\up'_{\l}<0$, integrating over $V$ again, we get the upper bound of $\bar{\up}_{\l}$,
\be\label{upb}
\bar{\up}_{\l}<-\f{1}{\ve V\ve}\iv u_0(x)d\mu.
\ee
Now, we need to prove that $\bar{\up}_{\l}$ has a lower bound.
 Integrating both sides of (\ref{cs2}) over $V$, we get
 \be\label{lb}\l \iv \me^{u_0+\up_{\l}}d\mu=\l \iv \me^{(b+1)(u_0+\up_{\l})}d\mu+4\pi N>4\pi N.\ee
 Using the  H\"{o}lder inequality and the Trudinger-Moser inequality, we calculate
 \be\label{lo}
\begin{split}
\iv \me^{u_0+\up_{\l}}d\mu &= \iv \me^{u_0+\bar{\up}_{\l}+\upl}d\mu\\
&\leq \me^{\bar{\up}_{\l}}\max\limits_{x\in V}\me^{u_0}\iv \me^{\up'_{\l}}d\mu\\
&\leq C \me^{\bar{\up}_{\l}}\iv \me^{\V \na\upl\V_2\f{\up'_{\l}}{\V \na\upl\V_2}}d\mu\\
&\leq C \me^{\bar{\up}_{\l}}\iv\me^{\V \na\upl\V_2^2+\f{|\up'_{\l}|^2}{4\V \na\upl\V_2^2}}d\mu\\
&\leq  C \me^{\bar{\up}_{\l}}\me^{\V\na \upl\V_2^2}.
\end{split}
\ee
In view of  (\ref{lb}) and (\ref{lo}), we have

$$\me^{\bar{\up}_{\l}}\geq C\l^{-1}\me^{-\V \na\upl\V_2^2}$$
which together with (\ref{na}) and (\ref{upb}) yields
$$\ve \bar{\up}_{\l}\ve\leq C(1+\l+\l^2).$$
Furthermore,
$$\V \up_{\l}\V_{W^{1,2}(V)}\leq C(1+\l+\l^2).$$
\end{proof}

\begin{lemma}\label{cre}
The equation (\ref{cs2}) at $\l=\l_c$ admits a solution.

\end{lemma}
\begin{proof}
Suppose $\l_c<\l<\l_c+1$. Then we see  that $\{\up_{\l}\}$ has a uniform bound in $W^{1,2}(V)$ by Lemma 2.8. The space  $W^{1,2}(V)$ is precompact since it is  finite dimensional.   Noting the monotonicity of $\{\up_{\l}\}$ with respect to $\l$, we conclude that there is a function $\up_*$ in $W^{1,2}(V)$ such that

$$\up_{\l}\ra \up_*$$
as $\l\ra \l_c$ and the convergence is pointwise. The fact $u_0+\up_{\l}<0$ implies $u_0+\up_{*}<0$

By the convergence of $\{\up_{\l}\}$, we have
$$\D \up_{\l}\ra \D\up_*$$
and
$$\l\me^{u_0+\up_{\l}}\le(\me^{b(u_0+\up_{\l})}-1\ri)\ra \l_c\me^{u_0+\up_{*}}\le(\me^{b(u_0+\up_{*})}-1\ri)$$ as $\l\ra \l_c$.
Hence  $\up_*$ is a solution of (\ref{cs2}). This completes the proof.
\end{proof}
\section{Multiple solutions}

Define the functional
\be J(\up)=\f{1}{2}\int_V\ve \na \up\ve^2d\mu+\f{\l}{b+1}\iv\me^{(b+1)(u_0+\up)}d\mu-\l\iv\me^{u_0+\up}d\mu+\f{4\pi N}{\ve V\ve}\int_V \up d\mu.\ee

Indeed, every critical point of $J$ in $W^{1,2}(V)$ serves as a solution to (\ref{cs2}). In what follows, we employ the variational method to validate the second part of Theorem 1.1. We start by demonstrating that if $\l > \l_c$, (\ref{cs2}) provides a solution, denoted as $w_{\l}$, which forms a local minimum for $J(\up )$. This leads to the emergence of multiple solutions, if the maximal solution $u_{\l}$ that we derived in Lemma \ref{set} is distinct from $w_{\l}$. In the case where $u_{\l}$ equals $w_{\l}$, we will invoke the mountain pass theorem to unearth the second solution.

Employing Lemma \ref{cre}, we ascertain a solution for Equation (\ref{cs2}) via a variational approach.
\begin{lemma}
For every $\l>\l_c$, there exists a solution of  (\ref{cs2}) in $W^{1,2}(V)$. Such a solution  is a local minimum  point of $J$.
\end{lemma}
\begin{proof}
For every $\l>\l_c$, define
\be\label{dsb}
\Sigma=\{\up\in W^{1,2}(V)|\,\up\geq \up_*\,\, \text{in}\,\,V\},
\ee where $\up_*$ is  the solution of (\ref{cs2}) at $\l=\l_c$.
By the Young inequality, we have
$$\me^{u_0+\up}\leq \f{\me^{(b+1)(u_0+\up)}}{b+1}+\f{b}{b+1}.$$
Hence,
\be\label{jes}J(\up)\geq \f{1}{2}\int_V\ve \na \up\ve^2d\mu-\f{b\l}{b+1}\ve V\ve+\f{4\pi N}{\ve V\ve}\int_V \up d\mu.\ee
We see that  $J$ has a lower bound on $\Sigma$.

Denote
\be\label{inf}
\eta_0=\inf\{J(\up)|\up\in \Sigma\}.
\ee
Next, we prove that there exists $ w_{\l}\in \Sigma$ such that $J(w_{\l})=\eta_0$.

There is  a sequence $\{\up_n\}$   in $\Sigma$ such that
\be\label{jbd}J(\up_n)\rightarrow \eta_0.\ee
As in Lemma 2.8,  we write   $\up_n$ as $\up_n=\up_n'+\bar{\up}_n$,  where $\bar{\up}_{n}=\f{1}{\ve V\ve}\int_{V}\up_{n} d\mu$ and $\up'_{n}=\up_{n}-\bar{\up}_{n}$.
By (\ref{jbd}), we know that $J(\up_n)$ is bounded.
In view of  (\ref{jes}), we get
$$J(\up_n)+\f{b\l}{b+1}\ve V\ve\geq 4\pi N \bar{\up}_n,$$
which together with the fact $\up_n\geq \up_*$ gives the bound of $\bar{\up}_n$. 

Let's define $c_*=\inf\{\up_*(x)|x\in V\}$. Observing that the integral $\int_{V}(\up_n-c_*)d\mu$ is consistently bounded and realizing that $\up_n-c_*$ is always greater than or equal to zero, it follows that an upper bound exists for $\up_n$. This allows us to declare $\up_n$ as a bounded function. In addition, we can ascertain that the sequence ${\up_n}$ is indeed bounded in the functional space $W^{1,2}(V)$. 
 
 Since
$W^{1,2}(V)$
is precompact, we obtain
\be\label{lim}\up_n\ra w_{\l}\ee
as $n\ra \infty$ and $w_{\l}$ is a solution to problem (\ref{inf}). 

We use the same method as  in \cite{GT96} to prove that $w_{\l}$ is a solution of (\ref{cs2}).
 Define
$$w_t=\max\{w_{\l}+t\varphi, \up_*\}\in \Sigma,$$ where $\varphi\in W^{1,2}(V)$ and $t\in (0,1)$.
Letting $u_t=\max\{\up_*-(w_{\l}+t\varphi), 0\}$,  we write
$$w_t=w_{\l}+t\va+u_t.$$
Calculate
\bna
\begin{aligned}
0&\leq \f{1}{t}\le(J(w_t)-J(\wl)\ri)\\
&=\f{1}{2t}\int_V\le(\ve \na w_t\ve^2-\ve \na\wl\ve^2\ri)d\mu+\f{\l}{t(b+1)}\iv\le[\me^{(b+1)(u_0+w_t)}-\me^{(b+1)(u_0+\wl)}\ri]d\mu\\
&-\f{\l}{t}\iv\le(\me^{u_0+w_t}-\me^{u_0+\wl}\ri)d\mu+\f{4\pi N}{\ve V\ve t}\int_V \le(w_t-\wl\ri)d\mu \\
&=\f{1}{2t}\int_V\ve \na(t\va+u_t)\ve^2d\mu+\f{1}{t}\int_V \na \wl\cdot\na (t\va+u_t)d\mu+\f{\l}{t}\int_V \me^{u_0+\wl}\le(\me^{b(u_0+\wl)}-1\ri)t\va d\mu\\
&+\f{\l}{t(b+1)}\iv\le[\me^{(b+1)(u_0+w_t)}-\me^{(b+1)(u_0+\wl)}-(b+1)\me^{(b+1)(u_0+\wl)}t\va\ri]d\mu\\
&-\f{\l}{t}\iv\le(\me^{u_0+w_t}-\me^{u_0+\wl}-\me^{u_0+\wl}t\va\ri)d\mu+\f{4\pi N}{\ve V\ve}\le(\int_V \va d\mu+\f{1}{t}\int_V u_t d\mu\ri).
\end{aligned}
\ena
As a result,
\begin{align*}
\int_V &\na \wl\cdot\na \va d\mu+\l\int_V \me^{u_0+\wl}\le(\me^{b(u_0+\wl)}-1\ri)\va d\mu+\f{4\pi N}{\ve V\ve}\int_V \va d\mu\\
&\geq -\f{t}{2}\V \na \va\V^2_2-\int_V \na \va\cdot\na u_t d\mu-\f{1}{2t}\V \na u_t\V^2_2-\f{1}{t}\int_V \na \wl\cdot\na u_t d\mu\\
&-\f{\l}{t(b+1)}\iv\le[\me^{(b+1)(u_0+w_t)}-\me^{(b+1)(u_0+\wl)}-(b+1)\me^{(b+1)(u_0+\wl)}t\va\ri]d\mu\\
&+\f{\l}{t}\iv\le(\me^{u_0+w_t}-\me^{u_0+\wl}-\me^{u_0+\wl}t\va\ri)d\mu-\f{4\pi N}{\ve V\ve t}\int_V u_t d\mu.\\
&=O(t)+\f{1}{t}\int_V\na (-t\va-\wl+\up_*)\cdot\na u_t d\mu-
\f{1}{2t}\V \na u_t\V^2_2-\f{1}{t}\int_V \na \up_*\cdot \na u_t d\mu\\
&-\f{\l}{t}\int_{V} \me^{u_0+\up_*}\le(\me^{b(u_0+\up_*)}-1\ri)u_t d\mu-\f{4\pi N}{\ve V\ve t}\int_V u_t d\mu\\
&-\f{\l}{t(b+1)}\iv\le[\me^{(b+1)(u_0+w_t)}-\me^{(b+1)(u_0+\wl)}-(b+1)\me^{(b+1)(u_0+\wl)}(t\va +u_t)\ri]d\mu\\
&+\f{\l}{t}\iv\le[\me^{u_0+w_t}-\me^{u_0+\wl}-\me^{u_0+\wl}(t\va+u_t)\ri]d\mu.\\
&+\f{\l}{t}\iv\le[\me^{(b+1)(u_0+\up_*)}-\me^{(b+1)(u_0+\wl)}\ri]u_td\mu-\f{\l}{t}\iv \le(\me^{u_0+\up_*}-\me^{u_0+\wl}\ri)u_td\mu.
\end{align*}
Noting that $\up_*$ is the lower solution of (\ref{cs2}) if $\l>\l_c$, we obtain
\be
-\f{1}{t}\int_V \na \up_*\cdot \na u_td\mu-\f{\l}{t}\int_{V} \me^{u_0+\up_*}(\me^{b(u_0+\up_*)}-1)u_td\mu-\f{4\pi N}{\ve V\ve t}\int_V u_td\mu\geq 0.
\ee
Since
$$ \int_V \na(-t\va-\wl+\up_*)\cdot\na u_t\geq \V \na u_t\V^2_2,
$$ we have
\be
\f{1}{t}\int_V\na (-t\va-\wl+\up_*)\cdot\na u_t-
\f{1}{2t}\V \na u_t\V^2_2\geq 0.
\ee

We now put forth the proposition:
\begin{equation}\label{eul}
	|t\va+u_t|\leq t|\va|.
\end{equation}

If we consider the case where $u_t=0$, it directly results in
\begin{equation*}
	|t\va+u_t|\leq t|\va|.
\end{equation*}

In the instance where $u_t>0$, we find that both
\begin{equation*}
	w_{\l}+t\va<\up_*
\end{equation*}
and
\begin{equation*}
	t\va+u_t=t\va+\up_*-(w_{\l}+t\va)=\up_*-w_{\l}\leq 0.
\end{equation*}
are satisfied.

From these conditions, it follows that
\begin{equation*}
	|t\va+u_t|\leq w_{\l}-\up_*< -t\va=t|\va|.
\end{equation*}

Therefore, in all cases, it holds that
\begin{equation*}
	|t\va+u_t|\leq t|\va|.
\end{equation*}

It follows that
\be
\f{\l}{t(b+1)}\iv\le[\me^{(b+1)(u_0+w_t)}-\me^{(b+1)(u_0+\wl)}-(b+1)\me^{(b+1)(u_0+\wl)}(t\va +u_t)\ri]d\mu=O(t)
\ee
and \be \f{\l}{t}\iv\le[\me^{u_0+w_t}-\me^{u_0+\wl}-\me^{u_0+\wl}(t\va+u_t)\ri]d\mu=O(t)\ee
by the Taylor expansion.

Define
$$V_t=\{x\in V:\wl(x)+t\va(x)-\up_*(x)<0\}$$ and $$ V_0=\{x\in V:\wl(x)=\up_*(x)\}.$$
Noting (\ref{eul}) implies $|u_t|\leq 2t|\va|$, we have
\be\label{lel}
\begin{split}
&\le|\f{\l}{t}\iv\le[\me^{(b+1)(u_0+\up_*)}-\me^{(b+1)(u_0+\wl)}\ri]u_td\mu-\f{\l}{t}\iv \le(\me^{u_0+\up_*}-\me^{u_0+\wl}\ri)u_td\mu\ri|\\
&\leq \f{\l}{t}\int_{V_t\backslash V_0}2\le|\me^{(b+1)(u_0+\up_*)}-\me^{(b+1)(u_0+\wl)}-\me^{u_0+\up_*}+\me^{u_0+\wl}\ri||t\va|\\
&\leq C\int_{V_t\backslash V_0}|\va|\ra 0,
\end{split}
\ee
as $t\ra 0^+$.

Therefore, by the above inequalities, we get
\be\label{cri}
\int_V \na \wl\cdot\na \va+\l\int_V \me^{u_0+\wl}\le(\me^{b(u_0+\wl)}-1\ri)\va+\f{4\pi N}{\ve V\ve}\int_V \va\geq 0,
\ee
which implies  that for any $\va \in W^{1,2}(V)$, $\langle J'(\wl), \va\rangle\geq 0$.  The reverse inequality  holds if we  replacing $\va$ with $-\va$.  Hence,  $\wl$ is a critical point of $J$ in $W^{1,2}(V)$. We see that $\wl$ is also a solution of Eq.(\ref{cs2}).  

Noting (\ref{lim}) and $\{\up_n\}\subset \Sigma$, we have $\wl\geq \up_*$. Then we obtain
\be\label{moc}\begin{split}
\D (\up_*-\wl)&=\l_c\me^{u_0+\up_*}\le(\me^{b(u_0+\up_*)}-1\ri)-\l\me^{u_0+\wl}\le(\me^{b(u_0+\wl}-1\ri)\\
&=\l\me^{u_0+\up_*}\le(\me^{b(u_0+\up_*)}-1\ri)-\l\me^{u_0+\wl}\le(\me^{b(u_0+\wl)}-1\ri)\\
&\,\,\,\,\,+(\l_c-\l)\me^{u_0+\up_*}\le(\me^{b(u_0+\up_*)}-1\ri)\\
&>\l\le[\me^{u_0+\up_*}\le(\me^{b(u_0+\up_*)}-1\ri)-\me^{u_0+\wl}\le(\me^{b(u_0+\wl)}-1\ri)\ri] \\
&=\l\me^{\xi}\le[(b+1)\me^{b\xi}-1\ri]\le(\up_*-\wl\ri)\\
&\geq b\l\le(\up_*-\wl\ri),
\end{split}
\ee
where $u_0+\up_*\leq \xi\leq u_0+\wl$. If $\max\limits_{x\in V}(\up_*-\wl)(x)=(\up_*-\wl)(x_0)=0$  for some $x_0\in V$, then  we have
 $$\D(\up_*-\wl)(x_0)>0,$$ which is impossible.  Hence,  $\wl(x)>\up_*(x)$ for all $x\in V$.

Now we prove  that $w_{\l}$  a local minimum  point of the functional $J$ in $W^{1,2}(V)$.
 We argue by contradiction. Suppose that $w_{\l}$  is not a local minimum  point of  $J$. Then for any $n\geq 1$, we have

\be\label{inf2}
\inf\limits_{\V \up-w_{\l}\V_{W^{1,2}(V)}\leq \f{1}{n}}J(\up)<J(w_{\l}).
\ee
Similar to  problem (\ref{inf}),  the infimum of $J$ can be  achieved at some point $\up_n'\in \{\up:\V \up-w_{\l}\V_{W^{1,2}(V)}\leq \f{1}{n}\}$. Then we have
\be
\up_n'\ra w_{\l}
\ee
as $n\ra \infty$. Noting $w_{\l}>\up_*$, we see that  $\up_n'>\up_*$ and $\up_n'\in \Sigma$ for $n$ sufficiently large. Hence,
$$J(\up_n')\geq J(w_{\l}),$$  which  contradicts (\ref{inf2}).  We finish the proof.
\end{proof}

 Our hypothesis postulates that the maximal solution, denoted by $\up_{\l}$, acts as the local minimum point of the functional $J$. Our subsequent efforts will focus on demonstrating that the functional $J$ fulfills the Palais-Smale condition, using the line of argumentation presented in reference \cite{LZ}.
 
\begin{lemma}
Any sequence $\{\up_n\}\subset W^{1,2}(V)$ satisfiing

(1) $J(\up_n)\ra c$ as $n\ra \infty$,

(2) $\V J'(\up_n)\V\ra 0$ as $n\ra \infty$,

admits a convergent subsequence.

\end{lemma}
\begin{proof}
We have
\be\label{mp1}
\f{1}{2}\int_V\ve \na \up_n\ve^2d\mu+\f{\l}{b+1}\iv\me^{(b+1)(u_0+\up_n)}d\mu-\l\iv\me^{u_0+\up_n}d\mu+\f{4\pi N}{\ve V\ve}\int_V \up_n d\mu=c+o_n(1),
\ee
\be\label{mp2}
\le|\int_V\na \up_{n}\cdot\na\varphi+\l\int_V \me^{u_0+\up_n}\le(\me^{b(u_0+\up_n)}-1\ri)\varphi+\f{4\pi N}{\ve V\ve}\int_V\varphi\ri|\leq \epsilon_n\V\varphi\V_{W^{1,2}(V)}, \,\epsilon_n\ra 0,
\ee
as $n\ra \infty$. By taking $\varphi=-1$ in (\ref{mp2}), we conclude that there exists an positive integer $N$ such that if $n\geq N$,
\be\label{mp3}
0<c\leq \int_V \me^{u_0+\up_n}\le(1-\me^{b(u_0+\up_n)}\ri)\leq C,
\ee
 where $c$ and $C$ are constants depending only on $\l$.  Now  we  prove  that $\up_n(x)$ is bounded for any fixed $x\in V$ by contradiction. Suppose not, then we conclude that there exists $x'$ and $\{\up_{n_k}\}\subset\{\up_n\}$ satisfying
$\up_{n_k}(x')\ra +\infty\,(-\infty),$ as $k\ra +\infty$.
Without loss of generality, we assume that  $\up_{n_k}(x')$ tends to $ +\infty$ as $k\ra +\infty$.

From now on, we don't distinguish the sequence and the subsequence. We prove that there is a point $\bar{x}\in V$ such that $\{\up_{n_k}(\bar{x})\}$ is bounded. Supposing not, we obtain that for any $y\in V$,
$\up_{n_k}(y)$ tends to $ +\infty\,(-\infty),$
as $k\ra +\infty$,   which contradicts with (\ref{mp3}).

Noting $\up_{n_k}(x')\ra +\infty$, assume that

$$\Lambda_k=\max\limits_{x\in V}| \up_{n_k}(x)|=\up_{n_k}(x_k).$$ Then we get $\up_{n_k}(x_k)\ra +\infty $ as $k\ra \infty$.
In view of (\ref{mp1}), we have
\be\f{1}{2}\int_V\ve \na \up_{n_k}\ve^2d\mu+\f{\l}{b+1}\iv\me^{(b+1)(u_0+\up_{n_k})}d\mu-\l\iv\me^{u_0+\up_{n_k}}d\mu+\f{4\pi N}{\ve V\ve}\int_V \up_{n_k} d\mu=c+o_n(1).\ee
Noting (\ref{jes}) and denoting $M=\f{b\l}{b+1}\ve V\ve$, we obtain
\be c\geq \f{1}{2}\sum\limits_{x\in V}\sum\limits_{y\sim x}\omega_{xy}|\up_{n_k}(y)-\up_{n_k}(x)|^2-M-4\pi N \Lambda_k+o_n(1).\ee
Consequently,
\be\label{me1}
|\up_{n_k}(x)-\up_{n_k}(y)|\leq C\sqrt{c+M+4\pi N\Lambda_k} +o_n(1),
\ee
if  $x\sim y$.

Noting $G$ is finite and connected, we can find  a path connecting $\bar{x}$ and $x_k$, namely,
 \be\label{me2}\bar{x}=x_{1}\sim x_2\sim x_3\sim\cdot\cdot\cdot\sim x_{l}=x_k.\ee
Combining (\ref{me1}) and (\ref{me2}), we obtain
\be\label{me3}
\begin{split}
\up_{n_k}(\bar{x})&\geq  \up_{n_k}(x_k)-(l-1)\sqrt{c+M+4\pi N\Lambda_k}+o_n(1)\\
&=\Lambda_k-(l-1)\sqrt{c+M+4\pi N\Lambda_k}+o_n(1)\\
&\ra +\infty,\,\,\text{as}\,k\ra +\infty.
\end{split}
\ee
Since we have proved that $\{\up_{n_k}(\bar{x})\}$ is bounded, there is a contradiction. Hence,  $\up_n(x)$ is bounded for any fixed $x\in V$. Furthermore,  $\{\up_{n_k}(x)\}$ is bounded in $W^{1,2}(V)$ and there exists a function $\hat{\up}$ such that
$$\up_{n_k}(x)\ra \hat{\up}(x),$$  for all $x\in V$,  as $k\ra +\infty$. Then proof is completed.
\end{proof}
Next we use the mountain pass theorem to get the second solution  of (\ref{cs2}). Since  $\up_{\l}$ is  a  local minimum point for $J$ by the assumption, there exists a constant $r_0>0$ such that
\be\label{ine}
\inf\limits_{\V \up-\up_{\l}\V_{W^{1,2}(V)}=r_0} J(\up)>J(\up_{\l}).
\ee

For any $\tau>0$, we have
\be\begin{split}
J(\up_{\l}-\tau)-J(\up_{\l})&=\f{\l}{b+1}\iv\le[\me^{(b+1)(u_0+\up_{\l}-\tau)}-\me^{(b+1)(u_0+\up_{\l})}\ri]d\mu\\
&-\l\iv\le(\me^{u_0+\up_{\l}-\tau}-\me^{u_0+\up_{\l}}\ri)d\mu-4\pi N\tau\\
&\ra -\infty
\end{split}
\ee
as $\tau\ra +\infty.$  Hence, there exists a  constant $\tau_0$ such that $\V \up_{\l}-\tau_0\V_{W^{1,2}(V)}>r_0$ and
\be
J(\up_{\l}-\tau_0)<J(\up_{\l})-1<J(\up_{\l}).
\ee
Define  $$\Gamma=\{\gamma:[0,1]\ra W^{1,2}(V) \,\,\text{continuous} \,\,\gamma(0)=\up_{\l},\,\gamma(1)=\up_{\l}-\tau_0\}$$ and
\be
c_0=\inf\limits_{\gamma\in\Gamma }\sup\limits_{t\in[0,1]}J(\gamma(t)).
\ee
By (\ref{ine}),  we get $c_0>J(\up_{\l})$. Then the mountain pass theorem \cite{AR73} implies that  $c_0$ is a critical value of $J$. We obtain    the second solution of (\ref{cs2}),  which is  the critical point of $J$.

\vskip 30 pt
\noindent{\bf ACKNOWLEDGMENTS}

This work is  partially  supported by the National Natural Science Foundation of China (Grant No. 11721101), and by National Key Research and Development Project SQ2020YFA070080.
\vskip 8pt
\noindent{\bf AUTHOR DECLARATIONS}
\vskip 8 pt
\noindent{\bf Conflict of Interest}

The authors have no conflicts to disclose.
\vskip 8 pt
\noindent{\bf Author Contributions}

 {\bf Jia Gao}: Investigation(equal);  Writing- original draft (equal). {\bf Songbo Hou}: Methodology(equal);  Supervision(equal);  Writing-review \& editing (equal).
\vskip 8 pt
\noindent{\bf DATA AVAILABILITY}

There are no new data which are created or analyzed in this study.

\end{document}